\documentclass{amsart}   
\usepackage{graphicx, amsmath, amssymb, amsthm, amscd}
\usepackage{epsf}
\usepackage{color}

\usepackage{enumerate}

\DeclareGraphicsExtensions{.jpg,.pdf,.mps,.png}
\newtheorem{theorem}{Theorem}[section]
\newtheorem{lemma}[theorem]{Lemma}

\newtheorem{corollary}[theorem]{Corollary}
\theoremstyle{definition}
\theoremstyle{question}

\theoremstyle{example}
\newtheorem{example}[theorem]{Example}

\theoremstyle{remark}

\numberwithin{equation}{section}

\begin{document}

\title[$1/\kappa$-homogeneous long solenoids]{$1/\kappa$-homogeneous long solenoids}

\author{Jan P. Boro\'nski}
\address[J. P. Boro\'nski]{Institute for Research and Applications of Fuzzy Modeling, National Supercomputing Center IT4Innovations, Division of the University of Ostrava, 30. dubna 22, 70103 Ostrava,
Czech Republic -- and -- Faculty of Applied Mathematics,
AGH University of Science and Technology,
al. Mickiewicza 30,
30-059 Krak\'ow,
Poland}
\email{jan.boronski@osu.cz}

\author{Gary Gruenhage}
\address[Gary Gruenhage]{Department of Mathematics and Statistics, Auburn University, AL 36849, USA}
\email{garyg@auburn.edu}

\author{George Kozlowski}
\address[George Kozlowski]{Department of Mathematics and Statistics, Auburn University, AL 36849, USA}
\email{kozloga@auburn.edu}



\begin{abstract}
We study nonmetric analogues of Vietoris solenoids.   Let $\Lambda$ be an ordered continuum, and let $\vec{p}=\langle p_1,p_2,\dots\rangle$ be a sequence of positive integers.  We define a natural inverse limit space $S(\Lambda,\vec{p})$, where the first factor space is the nonmetric ``circle" obtained by identifying the endpoints of $\Lambda$, and the $n$th  factor space, $n>1$, consists of  $p_1p_2\cdot\dots \cdot p_{n-1}$ copies of $\Lambda$ laid end to end in a circle. We prove that for every cardinal $\kappa\geq 1$, there is an ordered continuum $\Lambda$ such that $S(\Lambda,\vec{p})$ is $\frac{1}{\kappa}$-homogeneous; for $\kappa>1$, $\Lambda$ is built from copies of the long line. Our example with $\kappa=2$ provides a nonmetric answer to a question of Neumann-Lara, Pellicer-Covarrubias and Puga from 2005, and with $\kappa=1$ provides an example of a nonmetric homogeneous circle-like indecomposable continuum.  We also show that for each uncountable cardinal $\kappa$ and for each fixed $\vec{p}$, there are $2^\kappa$-many $\frac{1}{\kappa}$-homogeneous solenoids of the form $S(\Lambda,\vec{p})$ as $\Lambda$ varies over ordered continua of weight $\kappa$.  Finally, we show that for every ordered continuum $\Lambda$  the shape of $S(\Lambda,\vec{p})$
depends only on the equivalence class of $\vec{p}$ for a relation similar to one used to classify the additive subgroups of $\mathbb{Q}$.  Consequently,  for each fixed $\Lambda$, as $\vec{p}$ varies, there are exactly
$\mathfrak{c}$-many different shapes, where $\mathfrak{c}=2^{\aleph_0}$,  (and there are also exactly that many homeomorphism types) represented by $S(\Lambda,\vec{p})$.

\end{abstract}

\maketitle

\section{Introduction}

The present paper is concerned with nonmetric analogues of Vietoris solenoids. Recall that a 1-dimensional (Vietoris) \textit{solenoid} is a compact and connected space (i.e. \textit{continuum}) given as the inverse limit of circles, with $n$-fold covering maps as bonding maps. (The circle is a trivial example of a solenoid with $n=1$.) Solenoids were first considered by Vietoris \cite{Vi} for $n=2$,  defined as what would now be called the mapping torus of a homeomorphism of the Cantor set.
 Van Dantzig \cite{vD}
defined and extensively studied $n$-adic solenoids for all integers 
$n$ greater than 1.  He gave two constructions, one via the mapping torus and another by using a nested sequence of solid tori in $\mathbb{R}^{3}$, and mentioned in passing that the latter construction could also be used to construct spaces for arbitrary products $b_{\nu}=\prod_{i=1}^{\nu}n_{i}$ of integers 
$n$ greater than 1 rather than simply powers $n^{\nu}$.  See \cite{H-R} for references to van Dantzig's further work and an extensive discussion of  solenoidal groups and $\mathbf{a}$-adic solenoids, where $\mathbf{a}$ is a sequence of integers 
greater than 1.
When Steenrod \cite{Stnrd} used a 2-adic solenoid as the inverse limit of circles in an example, he credited Vietoris for having defined the same space, which suggests that this modern representation had already passed into folklore.
McCord \cite{Mc} introduced higher dimensional analogues of Vietoris solenoids, where the factor spaces in the inverse limit are connected, locally pathwise connected, and semi-locally simply connected spaces, and each bonding map is a regular covering map. Smale \cite{SS}
 used the construction of solenoids in terms of a descending sequence of solid tori, to show that they can arise as hyperbolic attractors in smooth dynamical systems. His results were extended to higher dimensions by R.F. Williams \cite{Wi}. Solenoids have been a major theme of research in algebraic and general topology, topological algebra, dynamical systems and the theory of foliations, as part of a wider class of \textit{matchbox manifolds} (see e.g. \cite{ClHu} for recent results).

Each Vietoris solenoid $S$ is homogeneous and circle-like, and if $S$ is not the circle then it is also indecomposable. Recall that  a space $S$ is \textit{homogeneous} if for any two points $x$ and $y$ in $S$ there is a homeomorphism $h:S\to S$ such that $h(x)=y$. A continuum is \textit{indecomposable} if it is not the union of two proper subcontinua, and it is \textit{hereditarily indecomposable} if all subcontinua are indecomposable. Bing \cite{Bi60} gave an equivalent form of the following definition for metric continua: $S$ is \textit{circle-like} if for each open cover $\mathcal{U}$ of $S$ there is a finite open refinement $\{U_1,\ldots,U_t\}$ that forms a \textit{circular chain}; i.e. $U_i\cap U_j\neq\emptyset $ if and only if $|i-j|\leq 1 (\operatorname{mod} t)$. Hagopian and Rogers \cite{HaRo} classified homogeneuous circle-like metric continua. According to their classification the following are the only such continua: pseudo-arc \cite{Bi2}, Vietoris solenoids, and solenoids of pseudo-arcs \cite{Bi3}, \cite{Ro}. The pseudo-arc was constructed by Moise in 1948, as an example of a continuum homeomorphic to each of its nondegenerate subcontinua, but distinct from the arc \cite{Mo}. Moise also gave a short alternative proof of Bing's result that the pseudo-arc is homogeneous \cite{Mo2}. Later Bing showed that all hereditarily indecomposable arc-like continua are homeomorphic \cite{Bi}, and therefore Moise's example is homeomorphic to a space first described by Knaster in 1922 \cite{Kn}. Bing also gave another characterization of the pseudo-arc, as a homogeneous arc-like continuum \cite{Bih}. A related, important circle-like metric continuum is Bing's pseudo-circle \cite{Bi}. Fearnley \cite{Fa1} and Rogers \cite{Ro2} independently showed that the pseudo-circle is not homogeneous (see also \cite{KuGa} for a short proof). Later, Kennedy and Rogers \cite{KeRo} proved that the pseudo-circle has uncountably many orbits under the action of its homeomorphism group. Sturm showed that the pseudo-circle is not homogeneous with respect to continuous surjections \cite{FS}. The following question arises naturally.

{\bfseries Question.} {\itshape What degrees of homogeneity can be realized on circle-like continua?}

To make the above question more precise, given a continuum $Y$ let $\operatorname{Homeo}(Y)$ denote its homeomorphism group. $Y$ is \textit{$\frac{1}{\kappa}$-homogeneous} if the action of $\operatorname{Homeo}(Y)$ on $Y$ has exactly $\kappa$ orbits, where $\kappa$ is a cardinal number. So homogeneous spaces are $\frac{1}{1}$-homogeneous and the smaller the $\kappa$ is, the more homogeneous $Y$ is. Consequently, the pseudo-circle is $\frac{1}{\kappa}$-homogeneous, for an uncountable cardinal $\kappa$. It is an open question as to whether $\kappa$ is equal to the cardinality of the real numbers. The last few years provided examples of $\frac{1}{\kappa}$-homogeneous circle-like metric continua for the case when $\kappa$ is a natural number. Neumann-Lara, Pellicer-Covarrubias and Puga \cite{NePP} gave an example of a decomposable circle-like $\frac{1}{2}$-homogeneous continuum. They asked (Question 4.10 \cite{NePP}) if there exists an indecomposable $\frac{1}{2}$-homogeneous circle-like continuum. Such an example (in fact a class of examples) was constructed by Pyrih and Vejnar \cite{PyVe}, and independently by the first author \cite{Bo}. Jim\'enez-Hern\'andez, Minc and Pellicer-Covarrubias \cite{RiPP} constructed a family of $\frac{1}{n}$-homogeneous solenoidal continua for every integer $n>2$. Topologically inequivalent examples were also given by the first author \cite{Bo2}, and a $\frac{1}{\omega}$-homogeneous solenoidal continuum was constructed there as well. A hereditarily decomposable, planar, $\frac{1}{2^{\omega}}$-homogeneous circle-like continuum is obtained by the identification of two endpoints in a singular arc-like continuum given by Ma\'ckowiak \cite{Ma}\footnote{We thank an anonymous referee for bringing this fact to our attention.}. In fact, Ma\'ckowiak's example is even more nonhomogeneous, in the sense that every orbit of the homeomorphism group consists of a single point. 

In this note we provide a nonmetric positive answer to the question from \cite{NePP}, and then we go on to construct a $\frac{1}{\kappa}$-homogeneous example with the same properties for any cardinal $\kappa$, finite or infinite. For $\kappa=1$, the construction yields a nonmetric homogeneous circle-like indecomposable  continuum. We denote our examples $S(\Lambda,\vec{p})$, as each depends on an ordered continuum $\Lambda$ and a sequence $\vec{p}$ of positive integers.  We also show that for each uncountable cardinal $\kappa$ and for each fixed $\vec{p}$, there are $2^\kappa$-many $\frac{1}{\kappa}$-homogeneous solenoids of the form $S(\Lambda,\vec{p})$ as $\Lambda$ varies over ordered continua of weight $\kappa$.    Finally, we show that for every ordered continuum $\Lambda$  the shape of $S(\Lambda,\vec{p})$
depends only on the equivalence class of $\vec{p}$ for a relation similar to one used to classify the additive subgroups of $\mathbb{Q}$.  Consequently,  for each fixed $\Lambda$, as $\vec{p}$ varies, there are exactly
$\mathfrak{c}$-many different shapes, where $\mathfrak{c}=2^{\omega}$,  (and there are also exactly that many homeomorphism types) represented by $S(\Lambda,\vec{p})$.

Our notation for ordinals, cardinals,  and ordinal arithmetic follows \cite{kunen}.   So, e.g., $\omega$ is the least infinite ordinal and also denotes the least infinite cardinal, $\omega_1$ is the least uncountable ordinal and cardinal, etc.  We use $\kappa$ and $\lambda$ to denote cardinals that may be infinite, while $i,j,k,l,m$, and $n$ denote finite integers (possibly negative).  Also, an ordinal is the set of its predecessors,
e.g., $\omega=\{n:n<\omega\}=\{n:n\in \omega\}$ is the set of natural numbers, and $\omega_1$ is the set of countable ordinals.  The set of positive integers is denoted by $\mathbb N$.

\section{The finite case}

An \emph{ordered continuum} is a nondegenerate compact connected space $\Lambda$ which has a linear order and whose topology is the order topology.  If $a, b \in \Lambda$, then $[a, \, b]$ is the set of all $x \in \Lambda$ such that $a \le x \le b$, and  $\partial \Lambda$ is the set consisting of the first and last points of $\Lambda$.  If $\Lambda_{1}$ and $\Lambda_{2}$ are ordered continua, let $\Lambda_{1} \vee \Lambda_{2}$ be the ordered continuum obtained from the disjoint union of $\Lambda_{1}$ and $\Lambda_{2}$ by identifying the last point of $\Lambda_{1}$ with the first point of $\Lambda_{2}$ preserving the given orders in  $\Lambda_{1}$ and $\Lambda_{2}$.

All of our examples have the following form.   Take an ordered continuum  $\Lambda$, and let $\Sigma$ be the ``circle" $\Lambda/\partial \Lambda$.
More generally, for $n\geq 1$, let $\Lambda^{(n)} =\Lambda_1\vee \Lambda_2\vee\dots \vee\Lambda_n$ where $\Lambda_i=\Lambda$ for all $1\leq i\leq n$, and let  $\Sigma^{(n)}=\Lambda^{(n)}/\partial \Lambda^{(n)}$.

For convenience, let the identified  endpoints of the copies of $\Lambda$ in  $\Sigma^{(n)}$ be labeled
$\infty_0,\infty_1,..., \infty_{n-1}$.   Also, for $x\in \Lambda\setminus\partial\Lambda$, let $\infty_i+x$ denote its copy in the $i^{th}$ copy of $\Lambda$ as a subset of $\Sigma^{(n)}$.

 If $m$ and $n$ are positive integers, there is a natural mapping $\phi^m_{n}:\Sigma^{(mn)}\rightarrow \Sigma^{(n)}$ defined by $\phi^m_n(\infty_i)=\infty_j$ and $\phi^m_n(\infty_i+x)=\infty_j+x$, where
$x\in \Lambda\setminus\partial\Lambda$ and $j=i$ mod $n$.   Note that $\phi_n^m$ is an $m$-fold covering map, precisely $$(\phi_n^m)^{-1}(\infty_j+x)=\{\infty_{j+kn}+x: k=0,1,..,m-1\},$$
and the same formula without $x$ holds as well.

Let $\vec{p}=\langle p_1,p_2,\dots\rangle$ be a sequence of integers greater than 1.
Let $$S(\Lambda,\vec{p})=\varprojlim \{\phi^{p_n}_{k(n)},\Sigma^{(k(n))},n\in\mathbb{N}\},$$
where $k(1)=1$ and $k(n)=p_1p_2...p_{n-1}$ for $n>1$.

\begin{theorem}\label{indecomposable}  For each $\vec{p}$,  $S(\Lambda,\vec{p})$ is an indecomposable circle-like continuum.  \end{theorem}

\begin{proof}

{\scshape Claim 1.} \textit{$S(\Lambda,\vec{p})$ is circle-like.}
It is an immediate consequence of Lemma 3.8 of Chapter X of \cite{EilSt}
that an inverse limit of circle-like continua is circle-like. Since each factor space $\Sigma^{(k(n))}$ is circle-like so is $S(\Lambda,\vec{p})$.

{\scshape Claim 2.} \textit{$S(\Lambda,\vec{p})$ is indecomposable.}

By contradiction suppose that there are two proper subcontinua $C$ and $G$ of $S(\Lambda,\vec{p})$ such that $C\cup G=S(\Lambda,\vec{p})$. Let $C_n$ and $G_n$ be the projections onto $\Sigma^{(k(n))}$ of $C$ and $G$ respectively such that $\Sigma^{(k(n))}\notin\{C_n,G_n\}$. Since $C_n$ (as well as $G_n$) is a proper subcontinuum of $\Sigma^{(k(n))}$ it is an arc (perhaps nonmetric) that is a proper subset of $\Sigma^{(k(n))}$. Note that $(\phi^{p_n}_{k(n)})^{-1}(C_n)$ and $(\phi^{p_n}_{k(n)})^{-1}(G_n)$ each consists of $p_n$ disjoint homeomorphic copies of $C_n$ and $G_n$ respectively.  Since $\pi_{n+1}(C)$ is connected it follows that it misses a component of $(\phi^{p_n}_{k(n)})^{-1}(C_n)$. The same is true about $\pi_{n+1}(G)$ with respect to $(\phi^{p_n}_{k(n)})^{-1}(G_n)$. Consequently $\pi_{n+1}(C\cup G)=\pi_{n+1}(C)\cup \pi_{n+1}(G)\neq \Sigma^{(n+1)}$ contradicting surjectivity of $\pi_{n+1}$.
\end{proof}

Recall that the {\it long line} is the space obtained by putting a copy of the open unit interval $(0,1)$  in between $\alpha$ and $\alpha+1$ for each $\alpha\in\omega_1$, and giving it the natural order topology.  We will refer to this line as the ``standard" long line.  The {\it long line of length $\kappa$} can be similarly defined for any ordinal $\kappa$ (though it shouldn't be considered ``long" if $\kappa<\omega_1$).  Note that long lines are locally compact.  The {\it closed long line of length $\kappa$} adds $\kappa$ as a compactifying point.  We will often use interval notation to denote these lines and subintervals thereof; e.g., $[0,\omega_1)$ is the standard long line, and $[0,\kappa]$ is the closed long line of length $\kappa$.

Before stating our next result, let us describe some natural autohomeomorphisms of $\Sigma^{(n)}$ when constructed from the ordered continuum $\Lambda$ as previously described. Note that for each $k=0,1,...,n-1$,    the ``rotation" $R_k$
of $\Sigma^{(n)}$ that maps $\infty_i+x$ to $\infty_j+x$, where $j=i+k$ mod $n$, is an autohomeomorphism of $\Sigma^{(n)}$.  Also, given a homeomorphism $S:\Lambda\mapsto \Lambda$ which leaves the endpoints fixed,
the map $\hat{S}_n$ of $\Sigma^{(n)}$ which applies $S$  to the interior of each of the $n$ copies of $\Lambda$  and leaves the points $\infty_i$, $i=0, 1, ...,n-1$ fixed is another autohomeomorphism.
It is easily checked that these autohomeomorphisms commute with the bonding maps.

\begin{theorem}\label{onehalf}  Let $\Lambda=[0,\omega_1]$ be the standard closed long line.  Then for each sequence $\vec{p}$ of positive integers,  $S(\Lambda,\vec{p})$ is $\frac{1}{2}$-homogeneous.  \end{theorem}

\begin{proof} Recall that  $\Sigma$ is obtained by indentifying the endpoints of $[0,\omega_1]$; let $\infty$ denote the collapsed point $\{0,\omega_1\}$.   By a reasoning similar to \cite{Bo}, we shall show that the two orbits are given by $O_1=\{\vec{s}\in S(\Lambda,\vec{p}): s_1\neq\infty\}$ and $O_2=\{\vec{s}\in S(\Lambda,\vec{p}): s_1=\infty\}$.

First suppose $\vec{s},\vec{w}\in O_1$.   Choose $\alpha<\omega_1$ larger than both $s_1$ and $w_1$.        Since $[0,\alpha]$ is a metric arc, there is an autohomeomorphism of $[0,\alpha]$ which maps $s_1$ to $w_1$ and keeps the endpoints fixed.  Clearly this autohomeomorphism extends to an autohomeomorphism $S$ of $\Lambda$ which maps $s_1$ to $w_1$.  Let $H_1=\hat{S}_1$.
For $n>1$, note that $s_n$ and $w_n$ are  points corresponding to $s_1$ and $w_1$, resp., in one of the copies of $\Lambda$ making up $\Sigma^{(k(n))}$, though they need not both lie in the same copy.   So we let $H_n$ be  $\hat{S}_{k(n)}$ followed by the appropriate rotation to take $s_n$ to $w_n$.  Since these homeomorphisms commute with the bonding maps, the sequence $H_1,H_2,...$ defines an autohomeomorphism of the inverse limit space which maps $\vec{s}$ to $\vec{w}$.

 If $\vec{s},\vec{w}\in O_2$, then for each $n$,  $s_n$ and $w_n$ are $\infty_i$ and $\infty_j$ for some $i,j<n$, whence an appropriate rotation of $\Sigma^{(k(n))}$  will send $s_n$ to $w_n$.
 So again we obtain an autohomeomorphism of the inverse limit space sending $\vec{s}$ to $\vec{w}$.

Finally, note that every point in $O_1$ is a point of first countability in $S(\Lambda,\vec{p})$, while no point of $O_2$ is a $G_\delta$.  Thus no autohomeomorphism sends a point of $O_1$ to $O_2$. It follows that these two sets are precisely the orbits.  \end{proof}

Let $\Lambda_1=[0,\omega_1]$, and let $\Lambda_1^-=\Lambda_1\setminus \{\omega_1\}$.  In other words, $\Lambda_1^-$ is the standard long line.    Now let $\Lambda_2^o=\mathbb Z\times \Lambda_1^-$ with the lexicographic order, where $\mathbb Z$ is the set of integers.
 Note that the point $(n+1,0)=l.u.b. \{(n,x):x\in \Lambda\}$ and so $(n+1,0)$ compactifies
$\{n\}\times \Lambda$.  So $\Lambda_2^o$ is a locally compact connected LOTS (linearly ordered topological space) with no first or last point. We
may think of $\Lambda_2^o$ as the real line with each open interval $(n,n+1)$, $n\in \mathbb Z$, replaced by the open long line $(0,\omega_1)$.  For convenience, we
denote the point $(n,x)\in \Lambda_2^o$  by $n+x$ and the point $n+0=(n,0)$ by $n$.  Now let $\Lambda_2=\Lambda_2^o\cup\{-\infty, \infty\}$ be the two point
compactification of $\Lambda_2^o$.

Given $\Lambda_n$, let $\Lambda_{n+1}$ be obtained from $\Lambda_n$ just like $\Lambda_2$ was obtained from $\Lambda_1$.  I.e., $\Lambda_{n+1}$ is the two point compactification of $\mathbb Z\times \Lambda_n^-$ with the lexicographic order, where $\Lambda_n^-$ is $\Lambda_n$ minus its right endpoint.  It will be helpful to consider the following translation map on $\Lambda_n$'s. For $x$ in $\Lambda_1^-$, let $i+x$ denote the point  $(i,x)$  in $\Lambda_2^o$.   Now on $\Sigma$, $T_k$ is the map which sends $i+x$ to $(i+k)+x$ and $\infty$ is fixed. On $\Sigma^{(n)}$, $T_k$ does the same on each copy of $\Lambda_2^o$  and keeps the other points (i.e., $\infty_0,...,\infty_{n-1}$) fixed. $T_k$ is defined analogously in the inductive step, where $\Lambda_{n+1}$ is the two point compactification of $Z\times \Lambda_n^-$.

We want to prove that $S(\Lambda_n,\vec{p})$ is $\frac{1}{n+1}$-homogeneous.   It will be helpful to first prove
the following lemma.  Note that 1-dimensional Vietoris solenoids (inverse limits of circles with p-fold covering maps as bonding maps) have a base that consists of open sets homeomorphic to $\mathbb C\times (0,1)$, where $\mathbb C$ is a Cantor set. It is easy to see that the nonmetric solenoids we consider will have a similar property, where $(0,1)$ is replaced by a basic open set (i.e., an arc) in $\Sigma$. For completeness sake we sketch a proof of this fact (in a more general form that has essentially the same proof).

\begin{lemma}\label{cantor} Let $S=\varprojlim \{\phi_{n},\Sigma^{(n)},n\in\mathbb{N}\}$ be an inverse limit of locally connected spaces in which the bonding maps are finite-to-one (but at least $2$-to-one) covering maps.  Suppose
also that each point $x\in \Sigma$ is contained in an open set $O$ which for all $n$ is evenly covered by the map $\phi_{n,1}=\phi_1\circ\phi_2\circ...\circ\phi_{n-1}:\Sigma^{(n)}\mapsto \Sigma$.
Then any point $\vec{x}\in S$ has a local open base of sets $U$ homeomorphic to $U_1\times \mathbb C$, where $U_1=\pi_1(U)$ is connected and $\mathbb C$ is the Cantor set.  \end{lemma}

\begin{proof}
 Let $\vec{p}\in S$, and suppose $U$ is an open subset of $S$ that contains $\vec{p}$. Without loss of generality, we may assume $U=\pi_i^{-1}(V)$ for some connected open  subset $V$ of $X_i$ containing $p_i$.  We may also assume that $U_1=\pi_1(U)=\phi_{(i-1,1)}(V)$ is evenly covered by $\phi_{n,1}$ for all $n$.
Let $C=\pi_i^{-1}(p_i)$.  Then $C$ is closed, and contained in $\Pi_{j>i}(\phi_{j-1,i}^{-1}(p_i))$, a product of finite sets.  Hence $C$ is totally disconnected.   Since $|\phi_{j-1}^{-1}(x)|\geq 2$ for any $x$ and any $j>i$, it is easy to check that $C$ has no isolated points.  Thus $C$ is a Cantor set.

 Define $h:\pi_i^{-1}(V)\mapsto V\times C$ by $h(\vec{x})=(x_i,q)$, where $q\in C$ is such that for each $j>i$, $q_j$ and $x_j$ are in the same slice of $V$ with respect to $\phi_{j-1,i}$.
 It is straightforward to check that $h$ is a homeomorphism.  Since $V\cong U_1$, we are done. \end{proof}

The solenoids constructed in this paper satisfy the conditions of Lemma~\ref{cantor}, where $O$ can be taken to be any proper arc contained in $\Sigma$.  The following corollary will help us show that certain points are not in the same orbit.

 \begin{corollary}\label{cantorcor} Let $S$ and $S'$  be as in Lemma~\ref{cantor}, and let $\vec{x}\in S$, $\vec{y}\in S'$.   If there is an homeomorhpism of $S$ amd $S'$ taking $\vec{x}$ to $\vec{y}$, then any neighborhood $U_1$ of $y_1$ in $S'$ contains a homeomorphic copy of some neighborhood $V_1$ of $x_1$ in $S$.
 \end{corollary}

\begin{proof} Let $h:S\rightarrow S'$ be a homeomorphism taking $\vec{x}$ to $\vec{y}$, and let $U_1$ be a neighborhood of $y_1$.  By Lemma~\ref{cantor}, there is a neighborhood $W$ of $\vec{y}$ such that $W\cong W_1\times \mathbb C$, where $W_1=\pi_1(W)$ is connected and contained in $U_1$.
Let $V$ be a neighborhood of $x$ such that $V\cong V_1\times \mathbb C$, where $V_1=\pi_1(V)$ is connected, and such that $h(V)\subset W$. Since the components of $W$ are all homeomorphic to $W_1$ and those of $V$ to $V_1$, it follows that $V_1$ is homeomorphic to a subset of $W_1\subset U_1$.  \end{proof}

 \begin{theorem} For each positive integer  $\kappa\geq 1$ and sequence $\vec{p}$ of integers $\geq 2$, $S(\Lambda_{\kappa},\vec{p})$ is $\frac{1}{\kappa+1}$-homogeneous.  \end{theorem}

\begin{proof}  The proof is by induction.  Theorem~\ref{onehalf} takes care of the case $\kappa=1$.   It will be helpful to look at the case $\kappa=2$ before describing the inductive step.   So, let $S=S(\Lambda_2,\vec{p})$, and let $\infty$ denote the point $\partial \Lambda$ in $\Sigma$.
We claim that the three orbits are
\begin{enumerate}\item[(1)] $O_1=\{\vec{x}\in S : x_1=\infty\}$;
\item[(2)] $O_2=\{\vec{x}\in S : x_1=k, k\in \mathbb Z\}$;
\item[(3)] $O_3=\{\vec{x}\in S : x_1\notin \{\infty\}\cup \mathbb Z\}$.
\end{enumerate}

If $\vec{x}\in O_3$, then $x_1$ corresponds topologically to a point ($>0$) in the long  line $\Lambda$, and so
does $x_n$ for all $n$ (indeed the $x_n$'s all correspond to the same point in $\Lambda$).   Thus $\vec{x}$ is a point
of first countability.  Furthermore, $x_1$ has a neighborhood contained in the long line, so $\vec{x}$ has a neighborhood
consisting of points of first countability.

If $\vec{x}\in O_2$, then $x_1$ corresponds to the compactifying point $\omega_1$ in  $[0,\omega_1]$, and
so does each $x_n$.  So no coordinate is $G_\delta$, and it easily follows that $\vec{x}$ is not a $G_\delta$ point in $S$.

If $\vec{x}\in O_1$, then $x_1=\infty$  and for $n>1$, $x_n=\infty_i$ for some $i=0,1,...,n-1$.  Note that $\infty_i
=\lim_{k\rightarrow \infty}\infty_{i-1}+k=\lim_{k\rightarrow -\infty}\infty_i+k$.   So each $x_n$ is $G_\delta$, and it follows that
$\vec{x}$ is $G_\delta$ and hence a point of first countability.   But every neighborhood of $\vec{x}$ contains a point with first
coordinate $k$ for some $k\in \mathbb Z$, i.e., a non-$G_\delta$-point in $O_2$.

It follows from the above discussion that no point in $O_i$, $i=1,2,3$, is in the orbit of a point in $O_j$, $j\neq i$.  So it remains to
prove that for each $i$, if $\vec{x},\vec{y}\in O_i$, then $\vec{y}$  is in the orbit of $\vec{x}$.

Suppose  $\vec{x},\vec{y}\in O_1$, i.e., $x_1=y_1=\infty$.  Let $H_1:\Sigma\rightarrow \Sigma$ be the identity.  Suppose $n\geq 2$.  Then there
are $i,j<k(n)$ such that $x_n=\infty_i$ and $y_n=\infty_j$.   Let $H_n:\Sigma^{(k(n))}\rightarrow \Sigma^{(k(n))}$ be the rotation $R_k$, where
$k=j-i$ mod $k(n)$.  Then $(H_n)_{n\in \mathbb N}$ defines an autohomeomorphism of $S$ that maps $\vec{x}$ to $\vec{y}$.

Suppose $\vec{x},\vec{y}\in O_2$.   Then $x_1=p, y_1=q$ for some $p,q,\in \mathbb Z$, so let $H_1=T_k$, where $k=q-p$.  Suppose  $n\geq 2$.
There are $i,j<k(n)$ such that $x_n=\infty_i+p, y_n=\infty_j+q$.  Let $H_n=R_l\circ T_k$, where $l=j-i$ mod $k(n)$.  Note that
$T_k$ will take      $x_n=\infty_i+p$ to $y'_n=\infty_i+q$, and then the rotation $R_l$ takes $y_n'$ to $y_n$; hence $H_n(x_n)=y_n$.   Since
both $R_l$ and $T_k$ commute with the bonding maps, so does $H_n$.   So again, $(H_n)_{n\in \mathbb N}$ defines an autohomeomorphism of $S$ that maps
$\vec{x}$ to $\vec{y}$.

Finally, suppose $\vec{x},\vec{y}\in O_3$.  Then $x_1=p+w, y_1=q+z$ for some $p,q,\in \mathbb Z$ and $w,z\in L\setminus \{0\}$.   There is an autohomeomorphism
$A$ of $\Lambda$ which maps $w$ to $z$.  Let $\hat{A}$ be the autohomeomorphism of $\Sigma$ as we described previously, i.e., which applies $A$ in every copy of $\Lambda$, and then let $H_1=T_k\circ \hat{A}$, where $k=q-p$.
If $n\geq 2$, then $x_n=\infty_i+p+w$, $y_n=\infty_j+q + z$ for some $i,j<k(n)$.   Then let $H_n=R_l\circ T_k\circ \hat{A}$, where $l=j-i$ mod $k(n)$.
Similar reasoning to the previous case shows that these $H_n$'s  define an autohomeomorphism of $S$ that maps
$\vec{x}$ to $\vec{y}$.

To prove that the $\kappa+2$-orbits we will define for $S(\Lambda_{\kappa+1},\vec{p})$ really are different requires us to define the following ``types" of points: let $p$ be a point in some $\Lambda_\kappa$ or $\Lambda_\kappa/\partial\Lambda_\kappa$. Then $p$ is of
\begin{enumerate} \item[(1)] Type $1$ if $p$ has a neighborhood $N$ such that every point of $N$ is a point of first countability;
\item[(2)]Type $2$ if $p$ is not $G_\delta$, and has a neighborhood $N$ such that $p$ is the only non-$G_\delta$ point in $N$;
\item[(3)]Type $3$ if $p$ is a limit point of Type $2$ points, and has a neighborhood $N$ such that $p$ is the only limit of Type $2$ points in $N$.
\item[($\kappa+1$)] Type $\kappa+1$, for $\kappa \geq 3$, if $p$ is a limit point of Type $\kappa$ points, and has a neighborhood $N$ such that $p$ is the only limit of Type $\kappa$ points in $N$.
\end{enumerate}

Now suppose $\Lambda_\kappa$ satisfies:
\begin{enumerate}
\item[(i)] $\Lambda_\kappa$ has points of Type $i$ for $i=1,2,...,\kappa+1$, and each point of $\Lambda_\kappa$ is one of these types;
\item[(ii)] The endpoints of $\Lambda_\kappa$ are the only points in $\Lambda_\kappa$ of Type $\kappa+1$;
\item[(iii)] If $-\infty<x,y<\infty$ are the same type in $\Lambda_\kappa$, then there is an autohomeomorphism $h$ of $\Lambda_\kappa$ mapping $x$ to $y$ and keeping the endpoints fixed.
\end{enumerate}

It is easily checked that $\Lambda_3$ satisfies the above conditions.  Now we show that if $\Lambda_\kappa$ satisfies these conditions, then $S=S(\Lambda_{\kappa+1},\vec{p})$ has orbits $O_i=\{\vec{x}\in S:x_1\textrm{ has Type }i\}$ for $i=1,2,...,\kappa+2$.

Recall that  $\Lambda_{\kappa+1}$ is the  the two point compactification of $\Lambda_{\kappa+1}^o=\mathbb Z\times \Lambda_\kappa^-$ with the lexicographic order, where $\Lambda_\kappa^-=\Lambda_\kappa\setminus \{\operatorname{max} \Lambda\} $.   $\Lambda_{\kappa+1}^o$ can be thought of as the real line with each open interval $(k,k+1)$ replaced by $\Lambda_\kappa\setminus \partial\Lambda$.  By condition $(ii)$ above,  the points $k=(k,\operatorname{min} \Lambda)$ are the only Type $\kappa+1$ points in $\Lambda_{\kappa+1}$, and note that this makes the endpoints of $\Lambda_{\kappa+1}$ the only Type $\kappa+2$ points in that space, and hence the point $\infty$ of the corresponding quotient space $\Sigma$ the only Type $\kappa+2$ point there.

Let $\vec{x}\in O_i$ and $\vec{y}\in O_j$ with $i\neq j$.   We show that  there is no autohomeomorphism of $S$ taking $\vec{s}$ to $\vec{y}$.  We may assume $i>j$. By Corollary~\ref{cantorcor}, every neighborhood of $y_1$ contains a copy of some neighborhood of $x_1$.   By an easy induction, every neighborhood of $x_i$ contains non-$G_\delta$ points.  So we have a contraction if $j=1$, since $y_1$ then has a neighborhood of all $G_\delta$ points.  Suppose $j>0$.  Another easy induction shows that every neighborhood of $x_i$ contains infinitely many points of Type $j$.  So again we have a contradiction since $y_1$ has a neighborhood $N$ with only one Type $j$ point.  It follows that no autohomeomorphism of $S$ maps $\vec{x}$ to $\vec{y}$.

Suppose now that $\vec{x}$ and $\vec{y}$ are in the same $O_a$.  If $a=\kappa+2$, then $x_1=y_1=\infty$, and every $x_m$ (resp. $y_m$) for $m\geq 2$ is $\infty_i$ (resp. $\infty_j$) for some $i,j<m$.  Thus the appropriate rotation $R_k$ will map $x_m$ to $y_m$ and commute with the bonding maps; it follows that there is an authohomeomorphism of $S$ mapping $\vec{x}$ to $\vec{y}$.  If $a=\kappa+1$, then $x_1$ and $y_1$ correspond to points in $\mathbb Z$ in $\Lambda_{\kappa+1}^o$.  This case is easily taken care of by the same argument as for $O_2$ in the case $\kappa=2$.

If $a<\kappa+1$, the reasoning is similar to that of $O_3$ above.  To wit, $x_1=p+w, y_1=q+z$ for some $p,q,\in \mathbb Z$ and $w,z\in \Lambda_\kappa\setminus \partial\Lambda$, where $w$ and $z$ are of the same type.
By assumption $(iii)$ on $\Lambda_\kappa$, there  is an autohomeomorphism
$A$ of $\Lambda_\kappa$ which maps $w$ to $z$ and fixes the endpoints.  Let $H_1=\hat{A}_1$ be the autohomeomorphism of $\Sigma$ which applies $A$ in every copy of $\Lambda_\kappa$ making up $\Lambda_{\kappa+1}^o$, and then let $H_1=T_k\circ \hat{A}$, where $k=q-p$.
If $m\geq 2$, then $x_m=\infty_i+p+w$, $y_m=\infty_j+q + z$ for some $i,j<k(m)$.   Then let $H_m=R_l\circ T_k\circ \hat{A}$, where $l=j-i$ mod $k(m)$.  Similar reasoning to previous cases shows that these $H_m$'s  define an autohomeomorphism of $S$ that maps
$\vec{x}$ to $\vec{y}$.

Finally, we need to prove that $\Lambda_{\kappa+1}$ satisfies the inductive conditions.  Condition (i) is easily seen, and condition (ii) was already mentioned. It remains to check (iii).   Suppose $x,y \in \Lambda_{\kappa+1}^o$ are of the same type.  If that type is $\kappa+1$, then $x=p, y=q$ for some $p,q,\in \mathbb Z$,  so we can let $h=T_k$, where $k=q-p$. If the type is $<\kappa+1$, then $x=p+w, y=q+z$ for some $p,q,\in \mathbb Z$ and $w,z\in \Lambda_\kappa^o$, where $w$ and $z$ are of the same type.
By assumption $(iii)$ on $\Lambda_n$, there  is an autohomeomorphism
$A$ of $\Lambda_n$ which maps $w$ to $z$ and fixes the endpoints.  Let $h=T_k\circ\hat{A}$, where $\hat{A}$ is the map which applies $A$ to each copy of $\Lambda_\kappa$ making up $\Lambda_{\kappa+1}$ and keeps other points fixed.  Then $h$ is an autohomeomorphism of $\Lambda_{\kappa+1}$ which maps $x$ to $y$ and keeps the endpoints fixed.  \end{proof}

We finish this section with our example for $\kappa=1$; i.e. a nonmetric homogeneous circle-like indecomposable continuum. Recall that, as mentioned in the Introduction, all Vietoris solenoids are homogeneous, and metric homogeneous circle-like continua were classified in \cite{Ro}, but there do not seem to be any results in the literature that would explicitly prove the classification incomplete in the nonmetric case. In addition, in \cite{GutekHagopian} Gutek and Hagopian asked if there exists a nonmetrizable circle-like homogeneous indecomposable continuum having only arcs for nondegenerate proper subcontinua. Our example satisfies all but the last mentioned property, and it should be clear that all of its proper subcontinua are homeomorphic to an order-homogeneous nonmetric arc.

\begin{example} There is an ordered continuum $\Lambda$ such that $S(\Lambda, \vec{p})$ is a nonmetric homogeneous indecomposable circle-like continuum.\end{example}

  \begin{proof} Let $\Lambda$ be any nonmetric ordered continuum which is {\it order-homogeneous},
   i.e.,  $\Lambda\cong [x,y]$ for any $x<y\in \Lambda$.\footnote{``Order-homogeneous" is a special case (for an ordered continuum) of ``hereditarily equivalent".}    Such spaces
  (with additional properties not relevant here) have been constructed by, for example, K.P. Hart and J. van Mill \cite{hartvanmill}.
  Using methods of this section, it is easy to see that the corresponding spaces $\Sigma^{(n)}$ and $S(\Lambda,\vec{p})$ are homogeneous.
    $S(\Lambda,\vec{p})$ is nonmetric because it admits a continuous surjection onto the nonmetric first coordinate,  and it is indecomposable and circle-like by Theorem~\ref{indecomposable}.  \end{proof}
In addition to the above example, in private communication, Michel Smith informed us that he conjectures the results in \cite{SmithStone} could be used to exhibit other nonmetric homogeneous circle-like indecomposable continua.

\section{$1/\kappa$-homogeneity for infinite $\kappa$}

In the section, we show that for any infinite cardinal $\kappa$, there is a $1/\kappa$-homogeneous indecomposable circle-like continuum.   These continua are also of the form $S(\Lambda,\vec{p})$ for some ordered continuum $\Lambda$, but now $\Lambda$ is going to be simply the long line of some ordinal length.  We can also think of these $\Lambda$ as being obtained by putting copies of the standard long line end to end some ordinal number of times.

Ordinal multiplication will be useful here, so we recall some basics (see, e.g., the first chapter of \cite{kunen} for an excellent sketch of ordinal arithmetic).  If $\alpha$ and $\beta$ are ordinals, then  $\alpha\cdot \beta$ is the ordinal whose order type is that of the ordinal $\alpha$ (recall an ordinal may be thought of as the set of its predecessors) laid end to end $\beta$ times.  Formally, we can define $\alpha\cdot \beta$ as the ordinal whose order type is that of $\beta\times\alpha$  with the lexicographic order.   Multiplication is not commutative, e.g., $\omega\cdot 2$ is equal to two copies of $\omega$ end to end; it is the same as $\omega+\omega$.  However, $2\cdot \omega$ is the ordinal $2$ laid end to end $\omega$ times; note that the resulting order type is $\omega$, so $2\cdot \omega =\omega$.   Ordinal exponentiation is defined inductively.   $\gamma^0=1$, $\gamma^{\beta+1}=\gamma^\beta\cdot \gamma$, and if $\alpha$ is a limit ordinal, $\gamma^\alpha=sup\{\gamma^\beta:\beta<\alpha\}$.   So, e.g., $\omega^2=\omega\cdot\omega$, $\omega^3=
\omega^2\cdot\omega$, etc., and $\omega^\omega=sup\{\omega^n:n<\omega\}$.   In particular, note that $\omega^\omega$ is a countable ordinal.

The following lemma will be useful to determine a lower bound for the number of orbits of some $S(\Lambda,\vec{p})$'s where $\Lambda$ is a long line of some ordinal length.

\begin{lemma}\label{omegaalpha}  Let $\delta=\omega_1\cdot\omega^\alpha$, where $\alpha$ is some ordinal, and consider the closed long line $[0,\delta]$.   Then for any $x<\delta$, the interval $[x,\delta]$ cannot be homeomorphically embedded in $[0,y]$ for any $y<\delta$.  \end{lemma}

  \begin{proof}  By induction on $\alpha$.  If $\alpha=0$, then $\delta=\omega_1\cdot\omega^0=\omega_1\cdot 1=\omega_1$, and the result is well-known and easy to prove (e.g., $\omega_1$ is not $G_\delta$ in $[0,\omega_1]$, but any point $<\omega_1$ is  $G_\delta$).  So suppose $\alpha>0$ and the result holds for any $\beta<\alpha$.

  {\it Case 1.  $\alpha$ is a limit ordinal. }  Suppose by way of contradiction that $h:[x,\delta]\rightarrow [0,y]$ is a homeomorphic embedding,
  where $y<\delta$.    Choose $\beta<\alpha$ such that $\gamma=\omega_1\cdot \omega^\beta$ is greater than $max\{x,y\}$.
  Then $h$ embeds $[x,\gamma]$ homeomorphically into $[0,y]$ with $y<\gamma$, contradicting the induction hypothesis.

  {\it Case 2.  $\alpha=\beta+1$.}  Let $\gamma=\omega_1\cdot \omega^\beta$, and note that $\delta=\omega_1\cdot \omega^\alpha=
  \omega_1\cdot \omega^{\beta+1}=\omega_1\cdot \omega^\beta\cdot \omega=\gamma\cdot \omega$.   So $[0,\delta)$ is the same as countably many
  copies of $[0,\gamma)$ laid end to end in order type $\omega$, and $\delta =sup\{\gamma\cdot n:n<\omega\}$.

    Suppose  $h:[x,\delta]\rightarrow [0,y]$ is a homeomorphic embedding,
  where $y<\delta$. Let $n<\omega$ be least such that $\gamma\cdot n\geq h(\delta)$. Then $h(\delta)> \gamma\cdot(n-1)$, so there is some $k\in \omega$ with
  $h([\gamma\cdot (k-1),\gamma\cdot k])\subset [\gamma\cdot (n-1),\gamma\cdot n)$.  But $[\gamma\cdot (k-1),\gamma\cdot k]\cong [0,\gamma]$ and
  $[\gamma\cdot (n-1),\gamma\cdot n)\cong [0,\gamma)$,
  so this contradicts the induction hypothesis.
  \end{proof}

\begin{corollary}\label{nohomeo}  Let $\Lambda$ be a closed long line of some ordinal length.  If $\vec{x},\vec{y}\in S(\Lambda,\vec{p})$, $x_1=\omega_1\cdot \omega^\alpha$, and
$y_1=\omega_1\cdot \omega^\beta$, with $\alpha\neq \beta$, then no autohomeomorphism of $S(\Lambda,\vec{p})$ maps $\vec{x}$ to $\vec{y}$.
\end{corollary}

\begin{proof}  Without loss of generality, assume $\alpha>\beta$.  By Corollary~\ref{cantorcor}, every neighborhood of $y_1$ must contain a homeomorphic copy of a neighborhood of $x_1$.   But this would contradict Lemma~\ref{omegaalpha}.   \end{proof}

The next result shows that for any infinite cardinal $\kappa$, there is a $1/\kappa$ homogeneous space of the form $S(\Lambda,\vec{p})$; in fact, taking $\Lambda$ to be the closed long line of length $\omega_1\cdot \omega^\kappa$ works.  If $\kappa=\omega$, then this says that taking countably many copies of the standard long line and lining them up in type $\omega^\omega$ works.  As noted above,  $\omega^\omega$ is a countable ordinal,
so $\omega_1\cdot \omega^\omega$ is an ordinal strictly between the cardinals $\omega_1$ and $\omega_2$.   So is $\omega_1\cdot \omega^{\omega_1}$,
but this can be simplified.  Indeed, for any uncountable cardinal $\kappa$, $\omega^\kappa=\kappa$.   (The reason: $\alpha<\beta$ implies $\omega^\alpha<\omega^\beta$, so $\omega^\kappa$ has to be at least $\kappa$, and one may show by induction that $\alpha<\kappa$ implies $\omega^\alpha<\kappa$, so it can't be more than $\kappa$.) So $\omega_1\cdot \omega^{\omega_1}=\omega_1\cdot\omega_1=\omega_1^2$, and hence the long line of this length is the same as the standard long line laid end to end $\omega_1$ times.   If $\kappa\geq \omega_2$, the formula can be simplified even more:  $\omega_1\cdot \omega^\kappa=\omega_1\cdot \kappa=\kappa$.   Now we will prove:

\begin{theorem} Let $\kappa$ be an infinite cardinal.  If $\Lambda$ is the closed long line of length $\omega_1\cdot \omega^\kappa$, then $S(\Lambda,\vec{p})$ is $1/\kappa$-homogeneous.   In particular, the long line of length $\omega_1\cdot \omega^\omega$ yields a $1/\omega$-homogeneous continuum, the one of length $\omega_1^2$ yields a $1/{\omega_1}$-homogeneous continuum, and for $\kappa\geq \omega_2$, the long line of length $\kappa$ yields a $1/\kappa$-homogeneous continuum.  \end{theorem}

\begin{proof}  Let $\kappa$ be an infinite cardinal, let $\Lambda$ be the closed long line of length $\delta=\omega_1\cdot \omega^\kappa$.  We will show that $S(\Lambda,\vec{p})$ is $1/\kappa$-homogeneous; the ``In particular..." then follows by the remarks in the preceding paragraph.  That $S(\Lambda,\vec{p})$ has
at least $\kappa$ many orbits is immediate from Corollary~\ref{nohomeo}.

It remains to show that there are no more than $\kappa$ many orbits.  It is easy to see that one may use the rotations $R_j$ as previously defined  to show that if $x_1=y_1$, then $\vec{x}$ and $\vec{y}$ are in the same orbit.   Let $NG$ be the set of all points $x\in \Sigma$ that are either  not $G_\delta$ in $\Sigma$ or are a limit of non-$G_\delta$ points.  For each $x\in NG$, $$O_x=\{\vec{x}\in S(\Lambda,\vec{p}):x_1=x\}$$ is either an orbit or is properly contained in one. (We don't know which, but conjecture the former.)   Note that $NG$ is closed in $\Sigma$, and that
$\Sigma\setminus NG$ breaks up into $\kappa$ many components each homeomorphic in a natural way to the standard open long line $(0,\omega_1)$.  Indeed, note that if $\alpha\in NG$, $\alpha<\kappa$, then $\alpha+\omega_1$ is the least point in $NG$ greater than $\alpha$.   It follows that
$\Sigma\setminus NG$ is equal  to

$$\bigcup\{(\alpha,\alpha+\omega_1):\alpha<\kappa,  \alpha\in NG,\text{ or }\alpha=0\}.$$

If $x_1$ and $y_1$ fall
into the same maximal interval $(\alpha, \alpha+\omega_1)$  of $\Sigma\setminus NG$, then there is an autohomeorphism $h$ of $[\alpha,\alpha+\omega_1]$
 sending $x_1$ to $y_1$ (and leaving $\alpha$ and $\alpha+\omega_1$ fixed, as it must).  Note that each $x_n$ and $y_n$ correspond to $x_1$ and $y_1$ in one of the $k(n)$ copies of $\Lambda$ making up $\Sigma^{(k(n))}$.   So for each $n$,   an autohomeomorphism of $\Sigma^{(k(n))}$ which applies $h$ to the interval $(\alpha,\alpha+\omega_1) $ in each copy of $\Lambda$  in $\Sigma^{(k(n))}$, and leaves other points fixed, followed by the appropriate rotation, will take $x_n$ to $y_n$ and commute with the bonding maps.   The resulting
homeomorphism of $S(\Lambda,\vec{p})$ takes $\vec{x}$ to $\vec{y}$.

Thus the following are either orbits or proper subsets of an orbit:
\begin{enumerate}\item[(i)] For each $\alpha\in NG$, $O_\alpha=\{\vec{x}\in S(\Lambda,\vec{p}): x_1=\alpha\}$;
\item[(ii)] For each maximal connected interval $(\alpha,\alpha+\omega_1)$ in $\Sigma\setminus NG$, $P_\alpha=\{\vec{x}\in S(\Lambda,\vec{p}): x_1\in (\alpha,\alpha+\omega_1)\}$.
\end{enumerate}
As $|NG|=\kappa$,  $ S(\Lambda,\vec{p})$ has at most $\kappa$ many orbits.
\end{proof}

{\bf Conjecture.} The sets given in (i) and (ii) above are precisely the orbits of $S(\Lambda,\vec{p})$.

\section{Constructing $2^\kappa$-many Nonhomeomorphic Examples}

In this section, we show that there are $2^\kappa$-many nonhomeomorphic spaces $S(\Lambda,\vec{p})$ for fixed $\vec{p}$, as $\Lambda$ varies over ordered continua of weight (= least cardinality of a base) $\kappa$, where $\kappa$ is any uncountable cardinal.  This is the best possible, because there are only $2^\kappa$-many compact  Hausdorff spaces of weight $\kappa$ (this follows easily from the fact that they all embed homeomorphically into $[0,1]^\kappa$).

Recall that our examples from the previous section are based on closed long lines $\Lambda(\kappa)$ of length $\omega_1\cdot \omega^\kappa$ (ordinal arithmetic), which can be thought of as laying the standard long line end to end in a sequence of order type $\omega^\kappa$.   The main idea for getting different homeomorphism types is to stick in a {\bf reverse} long line in place of some of those copies of the standard long line.

Let $a_\beta$ denote the ordinal $\omega_1\cdot \omega^\beta$.   Let $A$ be a subset of $\kappa$.  For each $\beta$ in $A$, in place of the copy of the standard long line from $a_\beta$ to $a_\beta + \omega_1$, put the reverse long line there instead.    Call the resulting ordered continuum $\Lambda(\kappa, A)$.    We will show that if $A$ and $A'$ are different subsets of $\kappa$, then $S(\Lambda(\kappa,A), \vec{p})$ is not homeomorphic to $S(\Lambda(\kappa,A'), \vec{p})$.

First, we need the following mild generalization of Lemma 3.1 which has a nearly identical proof.  We'll use the following notation: if $[x,y]$ is an interval in a long line, and $A$ is any set of ordinals, then
$[x,y]_A$ denotes the interval $[x,y]$ as modified above, i.e., for every $\beta\in A$ with $a_\beta\in [x,y)$, we have put the reverse long line in place of the copy of the standard long line from $a_\beta$ to $a_\beta + \omega_1$.

\begin{lemma}\label{omegaalpha2}  Let $a_\alpha=\omega_1\cdot\omega^\alpha$, where $\alpha$ is some ordinal.   Then for any $x<a_\alpha$, and for any sets $A$ and $B$ of ordinals, the interval $[x,a_\alpha]_A$ cannot be homeomorphically embedded in $[0,y]_B$ for any $y<a_\alpha$.  \end{lemma}

  \begin{proof}  By induction on $\alpha$.  The cases where  $\alpha=0$ and where $\alpha$ is a limit ordinal are handled just as in Lemma~\ref{omegaalpha}.   So we suppose $\alpha>0$ is a successor ordinal, say $\alpha=\beta +1$, and the result holds for any $\gamma<\alpha$.

  As in the proof of Lemma~\ref{omegaalpha},    $[0,a_\alpha)$ is the same as countably many
  copies of $[0,a_\beta)$ laid end to end in order type $\omega$, and $a_\alpha =sup\{a_\beta\cdot n:n<\omega\}$.
Suppose  $h:[x,a_\alpha]_A\rightarrow [0,y]_B$ is a homeomorphic embedding,
  where $x,y<a_\alpha$. Let $n<\omega$ be least such that $a_\beta\cdot n\geq h(a_\alpha)$.  Then $h(a_\alpha)>a_\beta\cdot(n-1)$.  If $h(a_\alpha)\leq a_\beta\cdot(n-1)+\omega_1$, then
    some arc $[z,a_\alpha]$ is homeomorphic either to a subset of a reverse long line (if $n=2$), or to the standard long line, both of which are impossible.  Hence
    $h(a_\alpha)> a_\beta\cdot(n-1)+\omega_1$, so there is $k\in\omega$,
    $k>2$, such that
  $h([a_\beta\cdot (k-1), a_\beta\cdot k]_A)\subset [a_\beta\cdot(n-1)+\omega_1,a_\beta\cdot n)_B$.  Since $k>2$, no reverse long lines have been inserted in $[a_\beta\cdot (k-1), a_\beta\cdot k]$, so $[a_\beta\cdot (k-1), a_\beta\cdot k]_A=[a_\beta\cdot (k-1), a_\beta\cdot k]\cong [0,a_\beta]$.  Also $[a_\beta\cdot (n-1)+\omega_1, a_\beta\cdot n)_B
  =[a_\beta\cdot (n-1), a_\beta\cdot n)\cong [0,a_\beta)$.  Hence the induction hypothesis is contradicted.
  \end{proof}

\begin{theorem} Fix $\vec{p}\in \mathbb{N}^\omega$.   Let $\kappa$ be an infinite cardinal, $A\subset\kappa$, and let $\Lambda(\kappa, A)$ be as defined above.  Then $S(\Lambda(\kappa, A),\vec{p})$ is $\frac{1}{\kappa}$-homogeneous, and $A\neq A'$ implies $S(\Lambda(\kappa, A),\vec{p})\not\cong S(\Lambda(\kappa, A'),\vec{p})$.   Hence, for any uncountable cardinal $\kappa$, there are $2^\kappa$-many $\frac{1}{\kappa}$-homogeneous solenoids of the form $S(\Lambda,\vec{p})$ as $\Lambda$ varies over ordered continua of weight $\kappa$.
\end{theorem}

\begin{proof}   With Lemma~\ref{omegaalpha2} in hand, the proof that $S(\Lambda(\kappa, A),\vec{p})$ is $\frac{1}{\kappa}$-homogeneous is the same as before.   So let $A$ and $A'$ be two distinct subsets of $\kappa$, and suppose $h:S(\Lambda(\kappa, A),\vec{p})\rightarrow S(\Lambda(\kappa, A'),\vec{p})$ is a homeomorphism.

Without loss of generality, there exists $\alpha\in A\setminus A'$.  Let $\vec{x}\in S(\Lambda(\kappa, A),\vec{p})$ with $x_1=\omega_1\cdot \omega^\alpha$, and let
$h(\vec{x})=\vec{y}$.  Since $\alpha\in A$, a reverse long line immediately follows $x_1$.   Let $b$ be any point in that reverse long line, and consider the interval $(0,b)$.   By Corollary~\ref{cantorcor}, some arc $(\delta, y_1+\epsilon)$, where $\delta<y_1$ and $0<\epsilon\leq 1$, is homeomorphic to a subarc of $(0,b)$.  Let $f:(\delta, y_1+\epsilon)\rightarrow (0,b)$ be a homeomorphic embedding.

We claim that $f(y_1)=x_1$.   If not, then for some subarc $I$ of $(\delta, y_1+\epsilon)$ containing $y_1$, either $f(I)$ is entirely on the right of $x_1$, or $f(I)$ is entirely on the left of $x_1$ and bounded by some $z<x_1$.    Invoking Corollary~\ref{cantorcor} in the opposite direction, some
arc $J$ containing $x_1$ is homeomorphic to a subarc of $I$.   If $f(I)$ is on the right of $x_1$, then $J$ is homeomorphic to a subarc of a reverse long line, which is impossible.   On the other hand, if $f(I)$ is bounded to the left of $x_1$, then Lemma~\ref{omegaalpha2} is violated.

Thus $f(y_1)=x_1$. Now $(y_1,sup(I))$ is homeomorphic to the real interval  $(0,\epsilon)$, so $x_1$ must look like that to its left or right, which is false.  So $f$, and hence the homeomorphism $h$, do not exist.
  \end{proof}

The reason for restricting the last sentence of the above result to uncountable cardinals is that the long line of length $\omega_1\cdot\omega^\omega$ has weight $\omega_1$, not $\omega$.  However, the first author shows in \cite{Bo2} that there are
$2^\omega$-many $\frac{1}{\omega}$-homogeneous solenoidal continua of weight $\omega$ (of course, having weight $\omega$ implies that they are  metrizable).

\section{Shape of spaces with related linear ordering}
In this section we show that among ordered continua $\Lambda$ the shape of $S(\Lambda,\vec{p})$
depends only on the equivalence class of $\vec{p}$ for a relation similar to one used to classify the additive subgroups of $\mathbb{Q}$.
We obtain a theorem about certain quotients of ordered continua and then apply this theorem to analyze the associated Bruschlinsky group, which is naturally isomorphic to the \v{C}ech-Alexander-Spanier cohomology group with coefficients $\mathbb{Z}$ in dimension one  (see \cite{HuHT}).  In our specific situation the Bruschlinsky group of $S(\Lambda,\vec{p})$  is the group $\mathbb Q (\vec{p})$ of all rationals  of the form $m/p_{1}\cdots p_{n}$ where  $m \in \mathbb{Z}$ and $n \in \mathbb{N}$.
  It will follow that for each fixed $\Lambda$ there are exactly $\mathfrak{c}$-many shapes and exactly $\mathfrak{c}$-many different homeomorphism types.

An exposition of shape theory for the results in this section can be found in \cite{MS} by Marde\v{s}i\'{c} and Segal.  Their approach requires some preparation, but Section 2.4 of Chapter 1 of  \cite{MS} shows the equivalence with a simpler approach given in \cite{KozSegL} and \cite{KozSegV}.  Nevertheless \cite{MS} is a very useful reference for material related to this section.
For an exposition of homotopy and cohomology Spanier \cite{Sp} is a standard useful reference, and all definitions and notation not explicitly given are found there.

\textbf{Elementary shape theory}.  If $X$ and $Y$ are spaces, which will always be taken as compact Hausdorff, and $f \colon X \to Y$ is a map, then $[f]$ is the homotopy class of $f$  and $\pi_{X}(Y)=[X,Y]$ is the set of all homotopy classes of maps from $X$ to $Y$. Let $\mathcal{W}$ be the category whose objects are all spaces which are dominated by polyhedra (see \cite{MS}, \cite{Mil} for the flexibility in handling $\mathcal{W}$) and whose morphisms are homotopy classes of maps between the objects.   $\pi_{X}$ will here be considered a covariant functor defined on  $\mathcal{W}$ taking values in the category of sets and functions.    A \emph{shape map} from $X$ to $Y$ will be defined as a natural transformation $\Phi \colon \pi_{Y} \to \pi_{X}$.  Note the reversal!  (Marde\v{s}i\'{c} and Segal use the phrase \emph{shape morphism}.)
If $f \colon X \to Y$ is the homotopy class of a map from $X$ to $Y$, it induces a natural transformation $f^{\#} \colon \pi_{Y} \to \pi_{X}$ which is defined for any space $W$ in $\mathcal{W}$  by $f^{\#}(g)=gf$.
If $f$ is a map, it will be convenient to write $f^{\#}$ for $[f]^{\#}$.

A shape map $\Phi$ from $X$ to $Y$ is a \emph{shape equivalence} if it is an invertible natural transformation, and for this to occur it is necessary and sufficient that $\Phi \colon \pi_{Y} (W) \to \pi_{X}(W)$ is a bijection for every object $W$ in $\mathcal{W}$.  A map $f$ will be called a shape equivalence if and only if $f^{\#}$ is a shape equivalence.  A space $X$ has \emph{trivial shape} if the map of $X$ to a one point space is a shape equivalence.  Note that a space $X$ has trivial shape if and only if every map defined on $X$ with target an object of $\mathcal{W}$ is homotopic to a constant map. 
\medskip

\textbf{Definitions}.
\begin{enumerate}
\item  If $\Lambda$ and $\Lambda^{\prime}$ are ordered continua, then a map $\gamma \colon \Lambda \to \Lambda^{\prime}$ is \emph{compliant} 
if $\gamma$ maps the first point of $\Lambda$ to the first point of $\Lambda^{\prime}$ and  the last point of $\Lambda$ to the last point of $\Lambda^{\prime}$.  No condition is imposed on how $\gamma$ maps the other points of $\Lambda$.
\item If $\Sigma$ is obtained from $\Lambda$ by collapsing $\partial \Lambda$ to a point with quotient map $\tau$, a map $\phi \colon \Sigma \to \mathbb{S}^{1}$ is \emph{standard} 
if
there is a compliant map $\gamma \colon \Lambda \to \mathbb{I}$ and $\phi = \sigma \gamma \tau^{-1}$, where $\mathbb{S}^{1}$ will be treated as obtained from $\mathbb{I}$ by collapsing $\partial \mathbb{I}$ with quotient map $\sigma \colon \mathbb{I} \to \mathbb{S}^{1}$.\end{enumerate}
\textbf{Remark}.   Expressions like $\sigma \gamma \tau^{-1}$ will be used only if they are single-valued and therefore define  functions.

The analysis of constructions on these spaces is based on the following result.

\begin{lemma}   If $\Lambda$ is an ordered continuum, then $\Lambda$ has trivial shape. \end{lemma}

\begin{proof}  $\Lambda$ has a cofinal family  of covers $\mathcal{U}$ whose nerves $K(\mathcal{U})$, have the property that $ | K(\mathcal{U}) | $, the space of $K(\mathcal{U})$, is homeomorphic to $\mathbb{I}$.   Now let $\varphi \colon \Lambda \to W$ be a map into an arbitrary object $W$ of $\mathcal{W}$.  A standard argument (Cf. \cite{Sp}, \cite{HuHT}) gives maps $\theta \colon \Lambda \to  | K(\mathcal{U}) | $ and $\psi \colon  | K(\mathcal{U}) |  \to W$ such that $\varphi \simeq \psi \theta$, where $K$ is the nerve of an open cover which may be taken from a cofinal family and whose space $|K|$  therefore may be taken to be homeomorphic to $\mathbb{I}$.   Thus $\varphi$ is homotopic to a constant map showing that the shape of $\Lambda$ is trivial. \end{proof}

There are two major tools of shape theory which will be used: collapsing sets of trivial shape (\cite{MS}) and the continuity property (Cor. 4.8 of \cite{LeeRay}).  For convenience we give the arguments which apply for our results.

\begin{lemma} If $A$ is a closed set of trivial shape in  $X$, then the projection map $f \colon X \to X/A$ is a shape equivalence.  Furthermore, if $Y$ is the quotient of  $X$ obtained by identifying a finite number $n$ of disjoint compact subspaces of trivial shape to $n$ points, then the projection $f \colon X \to Y$ is a shape equivalence.\end{lemma}

\begin{proof}
Let $W  \in \mathcal{W}$.  If $h \colon X \to W$, then it may be assumed that $W$ is a compact polyhedron.  The map $h$ restricted to $A$ is homotopic to a constant map; the homotopy extension property implies that there is a homotopy $h_{t} \colon X \to W, t \in  \mathbb{I}$, such that $h_{0}=h$ and $h_{1}$ maps $A$ to a single point.  $h_{1}$ defines a map $g \colon X/A \to W$, and $gf \simeq h$.  This shows $f^{\#} \colon [X/A,W] \to [X]$ is surjective.

To show $f^{\#}$ is injective suppose $g_{0}, g_{1} \colon X/A \to W$, and suppose there is a homotopy $h_{t} \colon X \to W$ such that $h_{j}=g_{j}f$ for $j=0,1$.  Let $w_{j}$ be the point $g_{j}f(A)$  for $j=0,1$, and consider the space $F$ of all maps $( \mathbb{I},0,1 ) \to ( W, w_{0}, w_{1})$ which is an ANR \cite{HuTR}.  The homotopy defines a map of $A$ into $F$, which is homotopic to a constant which is a map of $A \times \mathbb{I}  \to W$ such that $A \times \{t\}$ is mapped to a point for all $t$.  Altogether these homotopies define a map of $A \times  \mathbb{I} \times \mathbb{I}$ into $W$ with the properties that $(x,0,t) \mapsto h_{t}(x)$, $A \times \{(j,t)\} \mapsto w_{j}$, and $A \times \{(t,1)\} \mapsto$ a point.  Use this along with the assignments $(x,t,0) \mapsto h_t(x)$ and $(x,j,t) \mapsto g_jf(x)$  (where $x \in X$, $t \in \mathbb{I}$, and $j=0,1$)   to define a map
\[ X \times \mathbb{I} \times \{0\} \cup X \times \{0,1\} \times \mathbb{I} \cup A \times \mathbb{I} \times \mathbb{I} \to W \] The homotopy extension property gives an extension $G \colon X \times \mathbb{I} \times \mathbb{I} \to W$.  The map $(x,t) \mapsto G(x,t,1)$ maps $A \times \{t\}$ to a point for each $t$, which defines a homotopy $g_{0} \simeq g_{1}$ by passing to the quotient.

The second assertion follows by induction.
\end{proof}

Direct limits of sets and functions are discussed in \cite{Sp}.
\begin{lemma} Let $X_{0} \xleftarrow{f_{1}} X_{1}  \xleftarrow{f_{2}} X_{2} \cdots $ be an inverse sequence of compact Hausdorff spaces with inverse limit $X$ and maps $q_{n} \colon X \to X_{n}$.  Then for any $W \in \mathcal{W}$ the maps $q_{n}^{\#} \colon [X_{n}, W] \to [X,W]$ define a bijection  from the direct limit of the sequence $[X_{0}, W]  \xrightarrow{f_{1}^{\#}} [X_{1}, W]  \xrightarrow{f_{2}^{\#}} [X_{2}, W]  \cdots $ to $[X,W]$.
\end{lemma}

\begin{proof}  For each $n$ let $Y_{n}$ be the disjoint union of $X_{0}$, \dots, $X_{n}$ and let $g_{n} \colon Y_{n} \to Y_{n-1}$ map $X_{k}$ identically to itself for $0 \le k < n$ and map $X_{n}$ to $X_{n-1}$ via $f_{n}$.  The inverse limit $Y$ of this inverse sequence contains $X$ with the $X_{n}$ limiting down to $X$.  $W$ may be assumed to be a compact polyhedron, which may be considered a retract of an open subset of the Hilbert cube.  Thus, if $g \colon X \to W$ is given, there is an extension of $g$ over a neighborhood of $X$ in $Y$, and this implies there is $n_{0}$ such that for each $n \ge n_{0}$ there is a map $g_{n} \colon X_{n} \to W$ such that the maps $g_{n} q_{n}$ converge uniformly to $g$.  Because sufficiently close maps into $W$ are homotopic, there is $n_{1}$ such that for all $n \ge n_{1}$ there is a homotopy $g_{n} q_{n} \simeq g$.  It is, of course, sufficient for surjectivity that there is one such index.

To obtain injectivity consider the inverse sequence obtained by replacing each $X_{n}$ with $X_{n} \times \mathbb{I}$.  Thus if $g_{0}, g_{1} \colon X_{m} \to W$ are maps such that $g_{0} q_{m} \simeq g_{1} q_{m}$, assume for simplicity that $m=0$ and define a map $X \times \mathbb{I} \cup Y \times \{0, 1 \} \to W$ by the homotopy on $X \times \mathbb{I}$ and by mapping $(x, j) \in X_{n} \times \{ 0, 1 \}$ to $g_{j} f_{1} \cdots f_{n}(x)$.  By an argument similar to the above it follows that there is $n_{0}$ such that for each $n \ge n_{0}$ there is a homotopy $G_{n, \, t} \colon X_{n} \to W$, $t \in \mathbb{I}$, such that $G_{n, \, j}=g_{j} f_{1} \cdots f_{n}$ for $j=0,1$.  As above one such index suffices to show injectivity.
\end{proof}

\begin{lemma}\label{standard}  If $\Sigma$ be obtained from $\Lambda$ by collapsing $\partial \Lambda$, then there exist standard maps $\phi \colon \Sigma \to \mathbb{S}^{1}$ and if $\phi$ and $\phi^{\prime}$ are standard maps, then $\phi \simeq \phi^{\prime}$.    Furthermore, every standard map is a shape equivalence.\end{lemma}

\begin{proof}  Let $\tau \colon \Lambda \to \Sigma$ be the quotient map.  By Urysohn's Lemma there is a map $\gamma \colon \Lambda \to \mathbb{I}$ which maps the first point of $\Lambda$ to 0 and the last point of $\Lambda$ to 1 so that $\phi = \sigma  \gamma \tau^{-1}$ is a standard map.  If $\phi$ and $\phi^{\prime}$ are defined by $\gamma$ and $\gamma^{\prime}$ respectively, then by Tietze's Theorem the maps $\gamma$ and $\gamma^{\prime}$ are homotopic rel $\partial \Lambda$, which induces a homotopy $\phi \simeq \phi^{\prime}$.

Now employ $\mathbb{I} \vee \Lambda$ from which $\Sigma_{\mathbb{I}}$ arises by identification of the points of $\partial ( \mathbb{I} \vee \Lambda)$.  The compliant maps $\beta \colon  \mathbb{I} \vee \Lambda \to  \mathbb{I}$  and $\alpha \colon  \mathbb{I} \vee \Lambda \to  \Lambda$ which respectively collapse $\Lambda$ and $\mathbb{I}$ to points induce $g \colon  \Sigma_{\mathbb{I}} \to \mathbb{S}^{1}$ and  $f \colon \Sigma_{\mathbb{I}} \to \Sigma$.  These maps collapse sets of trivial shape to a point and hence are shape equivalences.  Since $\phi f$ and $g$ are standard maps, they are homotopic.  Hence for every object $W$ of $\mathcal{W}$,
\[f^\# \phi^\# = (f \phi)^\#=g^\# \!,\]
and since $f^\#$ and $g^\#$ are bijections, $\phi^\#$ is a bijection.  Hence $\phi$ is a shape equivalence.
\end{proof}

\begin{theorem}\label{multiplyn} Assume that
\begin{enumerate}[(a)]
\item $\Sigma$ and  $\Sigma^{\prime}$ are obtained from ordered continua  $\Lambda$ and $\Lambda^{\prime}$ by collapsing $\partial \Lambda$ and $\partial \Lambda^{\prime}$ respectively;
\item $a_{0}< a_{1}< \cdots < a_{n}$ are points of $\Lambda$ with $a_{0}$ the first point of $\Lambda$ and $a_{n}$ the last point of $\Lambda$;
 \item For $j=1, \dots, n$ there are compliant maps $\gamma_{j} \colon  \Lambda_{j} \to \Lambda^{\prime}$,
 where $\Lambda_{j}=[a_{j-1}, \, a_{j}]$;
\item $f \colon \Sigma \to \Sigma^{\prime}$ is defined by the equations $f(\tau(x))=\tau^{\prime}(\gamma_{j}(x))$ for all  $x \in \Lambda_{j}$  for $j=1, \dots, n$, where $\tau \colon \Lambda \to \Sigma$ and $\tau^{\prime} \colon \Lambda^{\prime} \to \Sigma^{\prime}$ are the respective quotient maps.
\end{enumerate}
Then for any standard maps $\phi \colon \Sigma \to \mathbb{S}^{1}$ and $\phi^{\prime} \colon \Sigma^{\prime} \to \mathbb{S}^{1}$ the maps $\phi^{\prime} f$ and $\mu_{n} \phi$ are homotopic, where the map $\mu_{n} \colon \mathbb{S}^{1} \to \mathbb{S}^{1}$ is the $n$th-power map $z \mapsto z^{n}$.
\end{theorem}

\begin{proof}
Let $\Sigma^{\prime}_{\mathbb{I}}$ be obtained from $\mathbb{I} \vee \Lambda^{\prime}$ by the identifying the points of $\partial ( \mathbb{I} \vee \Lambda^{\prime})$.  For $j=1, \dots, n$ let  $\Lambda_{j} = \mathbb{I} \vee [a_{j-1}, \, a_{j}]$.  Let $\Lambda_{\mathbb{I}}$ be these copies laid end to end with the last point of $\Lambda_{ j-1}$ identified with the first point of $\Lambda_{ j}$ for $j=2, \dots, n$.  Finally, $\Sigma_{\mathbb{I}}$ is obtained by identifying the points of $\partial \Lambda_{\mathbb{I}}$.  For $j=1, \dots, n$ there are compliant maps $\overline{\gamma_{j}} \colon  \Lambda_{j} \to \mathbb{I} \vee \Lambda^{\prime}$ defined by the identity of $\mathbb{I}$ and $\gamma_{j}$.  The map $f_{\mathbb{I}} \colon \Sigma_{\mathbb{I}} \to \Sigma^{\prime}_{\mathbb{I}}$ is defined by the equations $f(\overline{\tau}(x))=\overline{\tau}^{\prime}(\overline{\gamma_{j}}(x))$ for all  $x \in \Lambda_{j}$  for $j=1, \dots, n$, where $\overline{\tau} \colon \Lambda_{\mathbb{I}} \to \Sigma_{\mathbb{I}} $ and $\overline{\tau}^{\prime} \colon \mathbb{I} \vee \Lambda^{\prime}  \to \Sigma^{\prime}_{\mathbb{I}} $ are the respective quotient maps.

The diagram below is homotopy commutative.
\[ \begin{CD}
\mathbb{S}^{1} @<{\psi}<< \Sigma_{\mathbb{I}} @>{q}>> \Sigma @>{\phi}>> \mathbb{S}^{1} \\
@V{\mu_{n}}VV @V{f_{\mathbb{I}}}VV @V{f}VV  \\
\mathbb{S}^{1} @<{\psi^{\prime}}<< \Sigma_{\mathbb{I}}^{\prime} @>{q^{\prime}}>>  \Sigma^{\prime} @>{\phi^{\prime}}>>\mathbb{S}^{1}
 \end{CD}  \]
The maps $q$ and $q^{\prime}$ collapse a finite number of sets of trivial shape to points and hence are shape equivalences.
Since $\psi$ and $\phi q$ are standard maps, $\psi \simeq \phi q$ and therefore $\psi^{\#}$ and $q^{\#} \phi^{\#}$ are inverses and similarly for $\psi^{\prime \#}$ and $q^{\prime \#} \phi^{\prime \#}$.
This gives rise to the commutative diagram (with $\Phi$ defined by commutativity), and the  compositions in the horizontal rows are the identity.
\[ \begin{CD}
 [\mathbb{S}^{1}, \mathbb{S}^{1}] @>{\psi^{\#}}>{\approx}> [\Sigma_{\mathbb{I}}, \mathbb{S}^{1}] @<{q^{\#}}<{\approx}< [\Sigma, \mathbb{S}^{1}] @<{\phi^{\#}}<{\approx}< [\mathbb{S}^{1}, \mathbb{S}^{1}] \\
@A{\mu_{n}^{\#}}AA @A{f_{\mathbb{I}}^{\#}}AA @A{f^{\#}}AA @A{\Phi}AA  \\
[\mathbb{S}^{1}, \mathbb{S}^{1}] @>{\psi^{\prime \#}}>{\approx}> [\Sigma_{\mathbb{I}}^{\prime}, \mathbb{S}^{1}] @<{q^{\prime  \#}}<{\approx}<  [\Sigma^{\prime}, \mathbb{S}^{1}] @<{\phi^{\prime  \#}}<{\approx}< [\mathbb{S}^{1}, \mathbb{S}^{1}]
 \end{CD}  \]
Therefore $\Phi=\mu_{n}^{\#}$, and consequently,
\[ [ \phi^{\prime} f] = (\phi^{\prime} f)^{\#}(\mathbf{1}_{S^{1}})
 =f^{\#} \phi^{\prime  \#}(\mathbf{1}_{S^{1}})
 = \phi^{\#} \mu_{n}^{\#} (\mathbf{1}_{S^{1}})
 = (\mu_{n}  \phi)^{\#} (\mathbf{1}_{S^{1}})
 = [\mu_{n}  \phi]. \]
This concludes the proof. \end{proof}

Next we apply this theorem to our specific situation.  As before,
let $\Lambda$ be an ordered continuum and let $\Lambda^{(n)}$ be $n$ copies of $\Lambda$ laid end to end as $\Lambda_{1} \vee \Lambda_{2} \vee \cdots \vee \Lambda_{n}$ with the last point of $\Lambda_{j}$ identified with the first point of $\Lambda_{j+1}$ for $j=1, \dots, n-1$.  Let $\Sigma$ and $\Sigma^{(n)}$ be the corresponding quotient spaces obtained by identifying first and last points of $\Lambda$ and $\Lambda^{(n)}$ with respective quotient maps $\tau$ and $\tau^{(n)}$.  Let $f \colon \Sigma^{(n)} \to \Sigma$ be the map which maps each image of $\Lambda_{j}$ in $\Sigma^{(n)}$ in the order preserving way to $\Lambda$ and passing to the quotient $\Sigma$: precisely, for $j=1, \dots, n$ there is a commutative diagram.
\[ \begin{CD}
 \Lambda  @<{\approx}<< \Lambda_{j}\\
@V{\tau}VV  @VV{\tau^{(n)} \mid \Lambda_{j} }V \\
 \Sigma @<{f}<< \Sigma^{(n)}\\
\end{CD} \]
By Theorem~\ref{multiplyn}  for all compliant maps $\gamma \colon \Lambda \to \mathbb{I}$ and $\delta \colon \Lambda^{(n)} \to \mathbb{I}$  the resulting standard maps $\phi$ and $\psi$ are shape equivalences and the following diagram is homotopy commutative, where $\mu_{n}$ is the $n$th-power function $z \mapsto z^{n}$.
\[ \begin{CD}
\Sigma @<{f}<<  \Sigma^{(n)} \\
@V{\phi}VV   @V{\psi}VV\\
\mathbb{S}^{1} @<{\mu_{n}}<<    \mathbb{S}^{1}\\
\end{CD} \]
Now suppose $\vec{p}$ is a sequence $p_{1}, \, p_{2}, \, \dots$ of integers greater than 1.  For convenience put $p_{0}=1$ and for each $n \in \mathbb{N}$ let $\Lambda_{n}$ be $\Lambda^{(p_{0}p_{1}\cdots p_{n})}$ and let $\Sigma_{n}$ be $\Sigma^{(p_{0}p_{1}\cdots p_{n})}$.   The construction above produces a map $f_{n} \colon \Sigma_{n} \to \Sigma_{n-1}$ for each positive integer $n$ and thus produces an inverse sequence with the respective inverse limits $S(\Lambda,\vec{p})$ and $S(\mathbb{I},\vec{p})$.  Note that $S(\mathbb{I},\vec{p})$ is homeomorphic to the standard metric solenoid.  Furthermore, there are standard maps $\phi_{n} \colon \Sigma_{n} \to \mathbb{S}^{1}$ and the diagram
\[ \begin{CD}
 \Sigma_{0} @<{f_{1}}<<   \Sigma_{1} @<{f_{2}}<<   \Sigma_{2} @<{f_{3}}<< \cdots  S(\Lambda,\vec{p}) \\
 @V{\phi_{0}}VV @V{\phi_{1}}VV @V{\phi_{2}}VV \\
\mathbb{S}^{1} @<{\mu_{p_{1}}}<< \mathbb{S}^{1} @<{\mu_{p_{2}}}<< \mathbb{S}^{1} @<{\mu_{p_{3}}}<< \cdots S( \mathbb{I}, \vec{p})
\end{CD} \]
is homotopy commutative.  Since the ladder is only homotopy commutative, no map is given between the inverse limits.

\begin{corollary}\label{bru} Let $\vec{p}=\langle p_1,p_2,\dots\rangle$ be a sequence of integers greater than 1.  For every  ordered continuum $\Lambda$ the shape of $S(\Lambda,\vec{p})$ is the same as the shape of $S(\mathbb{I},\vec{p})$.  Furthermore, the Bruschlinsky group $[S(\Lambda,\vec{p}), \mathbb{S}^{1}]$ is isomorphic to $\mathbb Q(\vec{p})$.  \end{corollary}
\begin{proof}
Every $\phi_{j}$ is a shape equivalence.  Hence for any object $W$ of $\mathcal{W}$ there results a commutative ladder of sets of homotopy classes.
\[ \begin{CD}
 [\Sigma_{0}, W] @>{f_{1}^{\#}}>>   [\Sigma_{1}, W] @>{f_{2}^{\#}}>>   [\Sigma_{2}, W] @>{f_{3}^{\#}}>> \cdots [S(\Lambda,\vec{p}), W]  \\
 @A{\phi_{0}^{\#}}A{\approx}A @A{\phi_{1}^{\#}}A{\approx}A @A{\phi_{2}^{\#}}A{\approx}A @A{\Phi}AA\\
[\mathbb{S}^{1}, W] @>{\mu_{p_{1}}^{\#}}>> [\mathbb{S}^{1}, W] @>{\mu_{p_{2}}^{\#}}>> [\mathbb{S}^{1}, W] @>{\mu_{p_{3}}^{\#}}>> \cdots  [S( \mathbb{I}, \vec{p}), W]
\end{CD} \]
The right hand ends of these diagrams are the direct limit of the sequences they end.  This gives the definition of $\Phi$ and the fact that it is a  bijection is a consequence of the fact that all the $\phi_{j}^{\#}$ are bijections.
The group $[\mathbb{S}^{1}, \mathbb{S}^{1}]$ is isomorphic with $\mathbb{Z}$.  Using one such isomorphism the direct sequence obtained by setting $W=\mathbb{S}^{1}$ in the ladder above becomes a sequence of groups equal to $\mathbb{Z}$ with maps equal to the appropriate multiplication.  The direct limit of the sequence is $\mathbb{Q}(\vec{p})$ which is therefore isomorphic to
 $[ S(\Lambda,\vec{p}), \mathbb{S}^{1}]$.
\end{proof}
Keesling \cite{Kees} and Eberhart, Gordh, and Mack \cite{EberGorMa} have shown that every circle-like continuum has the shape of an abelian compact topological group.  Their results use a result of Scheffer \cite{Schef}, which uses deep properties of topological groups.  To complete the identification of the shape of $S(\Lambda,\vec{p})$  using these results the Bruschlinsky group must be calculated, and in doing this the proof of Corollary \ref{bru} makes the extra step of determining the shape while maintaining the elementary character of the arguments of this section.

The spaces $S(\Lambda, \vec{p})$ have been constructed from sequences $\vec{p}$ integers greater than 1.  Such sequences
can be compared by an equivalence relation which is suggested by the classification of (additive) subgroups of $\mathbb{Q}$.  The material from Group Theory which follows is based on the section ``Subgroups of $\mathbb{Q}$'' of Chapter 10 of \cite{Rt}.

Define the \emph{characteristic} of a sequence $\vec{p}$ to be the sequence $h=h(\vec{p})$ whose terms are indexed by the prime numbers in their natural  order and consist of nonnegative integers or the symbol $\infty$: $h_\pi$ is the total number times (possibly $\infty$) the prime number $\pi$ appears in the factorization of all the numbers $p_n$ of $\vec{p}$.  Sequences $\vec{p}$ and $\vec{q}$ are \emph{equivalent}, if their characteristics $h=h(\vec{p})$ and $k=h(\vec{q})$ satisfy the following conditions
\begin{enumerate}
\item $h_\pi=\infty$ if and only $k_\pi=\infty$, and
\item the set of all primes $\pi$ with $h_\pi \ne \infty$ and the set of all primes $\pi$ with $k_\pi \ne \infty$ differ by no more than a finite number of primes.
\end{enumerate}
This notion of equivalence corresponds to subgroups of $\mathbb{Q}$ having the same \emph{type} as defined in \cite{Rt}, and it is shown there that $\mathbb{Q}(\vec{p})$ and $\mathbb{Q}(\vec{q})$ are isomorphic if and only if $\vec{p}$ and $\vec{q}$ are equivalent.  In fact Rotman attributes the  concepts and results from abelian group theory which apply here to the 1914 dissertation of F. W. Levi.  In \cite{Mc} McCord codified an equivalence relation for sequences of prime numbers, a concept considered earlier by Bing \cite{Bi60}.  That concept leads to the same subgroups of $\mathbb{Q}$ as the one used in this paper.

The following known result \cite{Mc} is stated separately for convenience and because
 of the elementary character of the proof, which depends only on the material above and the translation to the present context of von Neumann's description of the character group of $\mathbb{Q}$, \cite{jvN} p.~477.
To fix the notation let $\Sigma(\vec{p})$ be the compact topological group which is the inverse limit of
$\mathbb{T} \xleftarrow{f_1}  \mathbb{T} \xleftarrow{f_2}  \mathbb{T} \xleftarrow{f_3} \cdots$,
where  $\mathbb{T}$ is the topological group of all complex numbers of modulus 1 under multiplication and $f_n \colon  \mathbb{T} \to  \mathbb{T}$ is defined by $f_n(z)=z^{p_n}$.

\begin{lemma}\label{vN}  Let $\vec{p}$ and $\vec{q}$ be infinite sequences of integers greater than 1.  The topological groups $\Sigma(\vec{p})$ and $\Sigma(\vec{q})$ are topologically isomorphic if and only $\vec{p}$ and $\vec{q}$ are equivalent.
\end{lemma}
\begin{proof}   If $\vec{p}$ and $\vec{q}$ are equivalent, it follows  from basic properties of character groups found in e.g.~\cite{H-R} that the isomorphism of $\mathbb{Q}(\vec{p})$ and $\mathbb{Q}(\vec{q})$ induces a topological isomorphism of their character groups.  The character group of $\mathbb{Q}(\vec{p})$ is seen to be $\Sigma(\vec{p})$ by the following observation: to each character $\phi$ of $\mathbb{Q}(\vec{p})$ associate the sequence $\vec{z}$ of terms $z_n$:
\[ z_{n} = \phi \left ( \frac{1}{p_{0} \cdots p_{n}} \right ) \in \mathbb{T}, \quad z_n=\phi \left ( \frac{p_{n+1}}{p_{0} \cdots p_{n+1}} \right )= z_{n+1}^{p_{n+1}}.\]
The sequence $\vec{z}$ is an element of $\Sigma(\vec{p})$, and the assignment $\phi \mapsto \vec{z}$ is a topological isomorphism (cf.~section 25.5 of \cite{H-R}).

Conversely, assume $\Sigma(\vec{p})$ and $\Sigma(\vec{q})$ are topologically isomorphic.  A direct argument along the lines of the proof for Corollary~\ref{bru}
 shows that
the Bruschlinsky groups $[ \Sigma(\vec{p}), \mathbb{S}^1]=\mathbb{Q}(\vec{p})$ and $[ \Sigma( \vec{q}), \mathbb{S}^1]= \mathbb{Q}(\vec{q})$ are isomorphic
and therefore $\vec{p}$ and $\vec{q}$ are equivalent.
\end{proof}

\begin{corollary}\label{SameShape}  Let $\Lambda$ and $\Lambda'$ be ordered continua and $\vec{p}$ and $\vec{q}$ sequences  of integers greater than one.  Then $S(\Lambda, \vec{p})$
and $S(\Lambda', \vec{q})$ have the same shape if and only if $\vec{p}$ and
$\vec{q}$ are equivalent.
\end{corollary}

\begin{proof}
If $\vec{p}$ and $\vec{q}$ are equivalent, Lemma~\ref{vN} implies that $\Sigma(\vec{p})$ and $\Sigma(\vec{q})$ are homeomorphic.  Since $S(\mathbb{I}, \vec{p})$ is homeomorphic to $\Sigma(\vec{p})$, the previous corollary implies that the shape of $S(\Lambda, \vec{p})$ is the same as the shape of $S(\Lambda', \vec{q})$.

Conversely, if the shape of $S(\Lambda, \vec{p})$ is the same as the shape of $S(\Lambda', \vec{q})$, then the Bruschlinsky groups $[ S(\Lambda, \vec{p}), \mathbb{S}^1]=\mathbb{Q}(\vec{p})$ and $[ S(\Lambda', \vec{q}), \mathbb{S}^1]= \mathbb{Q}(\vec{q})$ are isomorphic
and therefore $\vec{p}$ and $\vec{q}$ are equivalent.
\end{proof}
\begin{corollary} For each ordered continuum $\Lambda$ as $\vec{p}$ varies, there are exactly $\mathfrak{c}$ many shapes and exactly $\mathfrak{c}$ many homeomorphism types.
\end{corollary}
\begin{proof}  The cardinality of the set of equivalence classes of sequences $\vec{p}$ of integers greater than 1 is $\mathfrak{c}$.  Selecting one space $S(\Lambda, \vec{p})$ as a representative from each equivalence class of sequences gives a set of homeomorphism types with cardinality $\mathfrak{c}$.  Since the cardinality of the set of all homeomorphism types of the spaces $S(\Lambda, \vec{p})$ does not exceed the cardinality $\mathfrak{c}$ of the set of all such sequences, the result follows.
\end{proof}

\section{Acknowledgments}
Jan Boro\'nski's work was partially supported by the NPU II project LQ1602 ``IT4Innovations excellence in science'' provided by the M\v{S}MT. The author also gratefully acknowledges the partial support from the MSK DT1 Support of Science and Research in the Moravian-Silesian Region 2014 (RRC/07/2014).

The authors would like to thank the referees for valuable comments which improved the presentation of the paper.  

\bibliographystyle{amsplain}

\begin{thebibliography}{40}
\bibitem{Bi2} {\sc Bing, R.H. } {\em A homogeneous indecomposable plane continuum,} Duke Math. J. 15 (1948), pp. 729--742.
\bibitem{Bi} {\sc Bing, R.H. } {\em Concerning hereditarily indecomposable continua,} Pacific J. Math. 1 (1951), pp. 43--51.
\bibitem{Bi60} {\sc Bing, R.H.} {\em A simple closed curve is the only homogeneous bounded plane continuum that contains an arc,} Canad. J. Math., 12(1960), pp. 209--230.
\bibitem{Bih} {\sc Bing, R.H.} {\em Each homogeneous nondegenerate chainable continuum is a pseudo-arc,} Proc. Amer. Math. Soc. 10 1959 345--346.
\bibitem{Bi3} {\sc Bing, R.H. and F. B. Jones}, {\em Another homogeneous plane continuum,} Trans. Amer. Math. Soc.,
90(1959), pp. 171--192.
\bibitem{Bo}{\sc Boro\'nski, J. P.} {\em On indecomposable $\frac{1}{2}$-homogeneous circle-like continua,} Topology Appl. 160 (2013) pp. 59--62.
\bibitem{Bo2}{\sc Boro\'nski, J. P.} {\em On the number of orbits of the homeomorphism group of solenoidal spaces,} Topology Appl. 182 (2015) pp. 98--106.
\bibitem{ClHu} {\sc Clark, A. and Hurder, S.} {\em Homogeneous matchbox manifolds,} Trans. Amer. Math. Soc.  365  (2013),  no. 6, pp. 3151--3191.
\bibitem{vD} {\sc van Dantzig, D.} {\em Ueber topologisch homogene Kontinua,} Fund. Math. 15 (1930), pp. 102--125
\bibitem{EberGorMa} {\sc Eberhart, C., Gordh, G. R., Jr. and Mack, J.} {\em The shape classification of torus-like continua and (n-sphere)-like continua,} General Topology and Appl. 4 (1974), pp. 85--94.
\bibitem{EilSt} {\sc Eilenberg, S. and Steenrod, N.} {\em Foundations of algebraic topology}, Princeton University Press, Princeton, 1952.
\bibitem{Fa1}{\sc Fearnley, L.} {\em The pseudo-circle is not homogeneous,} Bull. Amer. Math. Soc. 75 (1969), pp. 554--558.
\bibitem{GutekHagopian} {\sc Gutek, A. and Hagopian, C. L.} {\em A nonmetric indecomposable homogeneous continuum every proper subcontinuum of which is an arc,} Proc. Amer. Math. Soc. 86 (1982), no. 1, pp. 169--172.
\bibitem{HaRo}{\sc Hagopian, C. L. and Rogers, J. T. Jr.} {\em A classification of homogeneous circle-like continua,}
Houston J. Math., 3(1977), pp. 471--474
\bibitem{hartvanmill} {\sc Hart, K. P. and van Mill, J.} {\em A method for constructing ordered continua,} Topology Appl. 21 (1985), no. 1, pp. 35--49.
\bibitem{H-R} {\sc Hewitt, E. and Ross, K. A} {\em Abstract Harmonic Analysis, Vol. I: Structure of topological groups. Integration
theory, group representations}, 
Die Grundlehren der mathematischen Wissenschaften, Bd. 115
Academic Press, Inc., Publishers, New York; Springer-Verlag, Berlin-G\"{o}ttingen-Heidelberg, 1963.
\bibitem{HuHT} {\sc Hu, S.-T.} {\em Homotopy Theory}, Academic Press Inc.,  New York,  1959.
\bibitem{HuTR} {\sc Hu, S.-T.} {\em Theory of Retracts}, Wayne State Univ. Press,  Detroit, 1965.
\bibitem{RiPP}{\sc Jim\'enez-Hern\'andez, R., Minc, P. and Pellicer-Covarrubias, P.} {\em A family of circle-like, $\frac{1}{n}$-homogeneous, indecomposable continua,} Topology Appl. 160 (2013) pp. 930--936.
\bibitem{Kees} {\sc Keesling, J.} {\em On the shape of torus-like continua and compact connected topological groups,} Proc. Amer. Math. Soc. 40 (1973), no. 1, pp. 297--302.
\bibitem{KeRo} {\sc Kennedy, J. and Rogers, J. T., Jr.} {\em Orbits of the pseudocircle,} Trans. Amer. Math. Soc. 296 (1986), no. 1, pp. 327--340.
\bibitem{Kn} {\sc Knaster, B.} {\em Un continu dont tout sous-continu est indecomposable,} Fund. Math. 3 (1922), pp. 247--286.
\bibitem{KozSegL} {\sc Kozlowski, G. and Segal, J.} {\em Locally well-behaved paracompacta in shape theory,} Fund. Math. 95 (1977), pp. 55--71
\bibitem{KozSegV} {\sc Kozlowski, G. and Segal, J.} {\em Local behavior and the Vietoris and Whitehead theorems in shape theory,} Fund. Math. 99 (1978), pp. 213--225
\bibitem{kunen} {\sc Kunen, K.} {\em Set Theory: An introduction to independence proofs}, North Holland, Amsterdam, 1979.
\bibitem{KuGa} {\sc Kuperberg, K.; Gammon, K.} {\em A short proof of nonhomogeneity of the pseudo-circle,} Proc. Amer. Math. Soc. 137 (2009), no. 3, pp. 1149-1152.
\bibitem{LeeRay} {\sc Lee, C. and Raymond, F.} {\em \v{C}ech extensions of contravariant functors,} Trans. Amer. Math. Soc.  133  (1968),  pp. 415--434.
\bibitem{Ma} {\sc Ma\'ckowiak, T.} {\em Singular arc-like continua,} Diss. Math. 257, 35 p. (1986).
\bibitem{MS} {\sc Marde\v{s}i\'{c}, S. and Segal, J.} {\em Shape Theory: The inverse system approach}, North Holland, Amsterdam, 1982.
\bibitem{Mc}{\sc McCord, C.} {\em Inverse limit sequences with covering maps,} Trans. Amer. Math. Soc.,
(114) pp. 197--209, 1965.
\bibitem{Mil} {\sc Milnor, J.} {\em On spaces having the homotopy of a CW-complex,} Trans. Amer. Math. Soc. 90 (1959), pp. 272--280.
\bibitem{Mo} {\sc Moise, E. E.} {\em An indecomposable plane continuum which is homeomorphic to each of its nondegenerate subcontinua,}
Trans. Amer. Math. Soc. 63, (1948). 581--594
\bibitem{Mo2} {\sc Moise, E. E.} {\em A note on the pseudo-arc,} Trans. Amer. Math. Soc. 67, (1949). 57--58
\bibitem{jvN} {\sc Neumann, J. v.} {\em Almost periodic functions in a group.~I,} Trans. Amer. Math. Soc. 36 (1934), no. 3, pp. 445--492.
\bibitem{NePP} {\sc Neumann-Lara, V., Pellicer-Covarrubias, P. and Puga, I.} {\em On $\frac{1}{2}$-homogeneous continua,} Topology Appl. 153 (2006), no. 14, pp. 2518--2527.
\bibitem{PyVe} {\sc Pyrih, P. and Vejnar, B.} {\em Half-homogeneous indecomposable circle-like continuum,} Topology Appl. 160 (2013), no. 1, pp. 56--58.
\bibitem{Ro2} {\sc Rogers, J.T., Jr.} {\em The pseudo-circle is not homogeneous,} Trans. Amer. Math. Soc. 148 (1970), pp. 417--428.

\bibitem{Ro} {\sc Rogers, J. T., Jr.} {\em Solenoids of pseudo-arcs,} Houston J. Math. 3 (1977), no. 4, pp. 531--537.

\bibitem{Rt} {\sc Rotman, J. J.} {\em An Introduction to the theory of groups, fourth ed.}, Springer-Verlag, New York, 1995.
\bibitem{Schef} {\sc Scheffer, W.} {\em Maps between topological groups that are homotopic to homomorphisms,} Proc. Amer. Math. Soc. 33 (1972), no. 2, pp. 562--567.
\bibitem{SS} {\sc Smale, S.} {\em Differentiable dynamical systems,} Bull. of the AMS, 73 (1967), pp. 747--817.
\bibitem{SmithStone} {\sc Smith, M. and Stone, J.} {\em On non-metric continua that support Whitney maps,} Topology Appl., Volume 170, 15 June 2014, pp. 63--85.
\bibitem{Sp} {\sc Spanier, E. H.} {\em Algebraic topology}, first corrected Springer ed., Springer-Verlag, New York, 1966.
\bibitem{Stnrd} {\sc Steenrod, N. E.} {\em Universal Homology Groups,}
Amer. J. Math. 58 (1936), pp. 661--701.
\bibitem{FS} {\sc Sturm, F.} {\em Pseudo-solenoids are not continuously homogeneous,} Topology Appl. 171 (2014) pp. 71--86.
\bibitem{Vi} {\sc Vietoris, L.} {\em \"Uber den h\"oheren Zusammenhang kompakter R\"aume und eine Klasse von zusammenhangstreuen Abbildungen,} Math. Ann. 97 (1927), pp. 454--472.
\bibitem{Wi} {\sc Williams, R.F.} {\em Expanding attractors,} Publ. Math. IHES, t. 43 (1974), pp. 169--203
\end{thebibliography}

\end{document}